%% file: main.tex
\documentclass[letterpaper, 10 pt, conference]{ieeeconf}  

\IEEEoverridecommandlockouts                              

\usepackage[utf8]{inputenc}
\usepackage{graphicx}
\usepackage[center]{caption}
\graphicspath{{.}}
\usepackage{color}
\usepackage{epsfig}
\usepackage{latexsym}
\usepackage{amssymb}

\usepackage{mathrsfs}
\usepackage{bbm, dsfont}
\usepackage{amsmath, amsfonts, outlines, mathtools, outlines, cancel, hyperref, stackengine, bbm, romannum}

\let\labelindent\relax
\usepackage{enumitem}
\setenumerate[1]{label=(\roman*)}

\usepackage{setspace}
\usepackage{subcaption}
\usepackage[ruled,vlined,linesnumbered]{algorithm2e}
\usepackage[english]{babel}
\usepackage[utf8]{inputenc}
\usepackage[T1]{fontenc}
\usepackage{newtxtext,newtxmath}
\usepackage{algorithmic}
\usepackage{scalerel} 

\newtheorem{theorem}{Theorem}

\newtheorem{assumption}{Assumption}
\newtheorem{lemma}{Lemma}

\newtheorem{envdef}{Definition}
\newtheorem{coro}{Corollary}

\newenvironment{aspprime}[1]
  {\renewcommand{\theassumption}{#1$^\prime$}%
   \addtocounter{assumption}{-1}%
   \begin{assumption}}
  {\end{assumption}}

\allowdisplaybreaks
\input{macros.tex}
\title{\LARGE \bf
On the Convergence Rates of A Nash Equilibrium Seeking Algorithm in Potential Games with Information Delays
}

\author{Yuanhanqing Huang$^{1}$ and Jianghai Hu$^{1}$
\thanks{This work was supported by the National Science Foundation under Grant No. 2014816 and No.2038410. }
\thanks{$^{1}$The authors are with the Elmore Family School of Electrical and Computer Engineering, Purdue University, West Lafayette, IN, 47907, USA 
        {\tt\small \{huan1282, jianghai\}@purdue.edu}}%
}

\begin{document}

\maketitle
\thispagestyle{empty}
\pagestyle{empty}

\begin{abstract}

This paper investigates the equilibrium convergence properties of a proposed algorithm for potential games with continuous strategy spaces in the presence of feedback delays, a main challenge in multi-agent systems that compromises the performance of various optimization schemes. 
The proposed algorithm is built upon an improved version of the accelerated gradient descent method.  
We extend it to a decentralized multi-agent scenario and equip it with a delayed feedback utilization scheme. 
By appropriately tuning the step sizes and studying the interplay between delay functions and step sizes, we derive the convergence rates of the proposed algorithm to the optimal value of the potential function when the growth of the feedback delays in time is subject to \Tblue{sublinear, linear, and superlinear} upper bounds. 
Finally, simulations of a routing game are performed to empirically verify \Tblue{our findings}. 
\end{abstract}

\section{INTRODUCTION}

In a Nash equilibrium (NE) seeking problem, given a collection of selfish players whose cost functions may depend on the actions of others, the designer aims to propose \Tblue{a set of local protocols}, by following which this group of players eventually settle down to a stable action profile where unilateral deviations are unprofitable. 
Among the vast literature on noncooperative games and NE-seeking problems, potential games are a subclass of games that admit a potential function such that any local optimum of the potential function coincides with a local NE of the game \cite{monderer1996potential}. 
Although restrictive and less comprehensive by their nature, potential games still find applications in a wide range of societal systems, such as power control in wireless communication \cite{tse2005fundamentals,mertikopoulos2010dynamic}, flow control and routing in networks \cite{altman2002nash, rosenthal1973class}, and thermal load management of autonomous buildings \cite{jiang2021game}, to name a few. 
The problems discussed above are often examined via the general framework of online learning, which unfolds as follows: 
at each iteration, this group of players individually select an action, the chosen actions trigger (first-order or zeroth-order) feedback information for each player and the process repeats \cite{cesa2006prediction, bubeck2012regret}. 

A major challenge in many cases of practical interest is the latency between taking actions and receiving feedback due to, e.g., the existence of non-negligible communication delays, the fact that it takes time for physical systems to generate corresponding rewards for actions taken, etc. 
\Tblue{For example, in portfolio management, it is common for investors to wait for a specific period of time to assess the return rate of their investment strategies \cite{bell1988game}. }
This challenge is exacerbated in multi-agent systems, where different agents are affected by different amounts of delay. 
To tackle the adverse effect of feedback delay, having a coordinator to enforce synchronous updates is a solution, but all too often prohibitive due to the extra energy and resources needed, especially when confronted with a tremendous number of agents and their concerns about privacy. 
In view of all these, our paper tries to answer the following questions: 
Can we design a protocol for a class of potential games such that even under large delays, the induced sequence of decisions can converge to \Tblue{the NE set} in a decentralized manner? 
If so, how is the convergence rate of the proposed protocol influenced by the large delays inside the multi-agent systems? 

\textit{Related Work:} 
As for the literature about NE seeking in potential games, \cite{heliou2017learning} focused on potential games with finite sets of actions and proved the convergence properties of the actual sequence of play under semi-bandit and bandit feedback oracles.
Vu et al. \cite{vu2021fast} restricted their attention to routing games and designed two elegant methods tailored to this specific setup that are robust against random noise and admit adaptive step sizes. 
Paccagnan et al. \cite{paccagnan2018nash} proposed a decentralized algorithm to compute Wardrop equilibria of potential games with coupled constraints, where a central coordinator takes care of the update of some global variables such as population averages and the dual variables for global constraints. 
In this work, we examine a general class of potential games that have continuous strategy spaces and design a protocol free of central coordinators.

The rate of convergence has been studied in a vast corpus of literature on solving optimization problems, variational inequalities, and learning in games. 
By conducting a set of time-varying augmentations in the dual space and mapping the results back to the primal space, the optimal accelerated algorithms can achieve a $O(1/k^2)$ rate for optimization problems \cite[Sec.~3.6]{Nesterov1983method}\cite{krichene2015accelerated, su2014differential}. 
In \cite{gao2022continuous}, Gao et al. established an exponential convergence of continuous-time actor-critic dynamics in solving potential games. 
Yet as suggested by \cite{diakonikolas2019approximate}, due to the existence of discretization errors, the discrete-time convergence rate of an algorithm can be very different from its continuous-time counterpart, and our focus here is the discrete-time rate, which is more practical for physical implementations. 

When taking feedback delays into consideration, Hsieh et al. \cite{hsieh2021multi} proposed adaptive strategies to minimize the regrets for multi-agent systems where there exists a sequence of time-varying loss functions and unbounded feedback delays. 
Zhou et al. \cite{zhou2021distributed} analyzed the convergence of asynchronous stochastic gradient descent in the presence of large delays, which happen frequently in parallel computing environments. 
The most relevant work is \cite{heliou20agradient}, where Heliou et al. considered a multi-player setup and introduced a no-regret gradient-free learning policy to strictly monotone games. 
This policy guaranteed convergence to NEs under the restriction that  feedback delays are sublinear. 

\textit{Our Contributions:} 
Our first contribution is to propose an algorithm that achieves a convergence rate of $O(1/k^2)$ \Tblue{measured by the potential function values} if the feedback can be received instantaneously and ensures \Tblue{the convergence} if the delay at each iteration $k$ can be upper bounded by some \Tblue{fraction power} function of $k$. 
The proposed algorithm leverages the improved accelerated gradient descent (AGD+) from \cite{cohen2018acceleration} as the backbone, adapts the centralized AGD+ to the scenario of multi-player potential games, and equips it with a delayed information utilization scheme. 
Then, our theoretical results suggest \Tblue{the feasible step sizes to employ} under different upper bounds for feedback delays, which include sublinear, linear, and superlinear bounds. 
Although the convergence of optimization and online learning algorithms under bounded or sublinear delays has been extensively studied \cite{heliou20agradient, tsitsiklis1986distributed, lian2015asynchronous}, far less has been discussed when it comes to delays beyond those. 
Furthermore, based on properly tuned step sizes, we investigate how the convergence rates of the proposed algorithm depend upon different types of delay bounds. 
\Tblue{Complete proofs of the main statements and some intermediate results are omitted due to the space limit, which are available in \cite{huang2022convergence}.}

\textit{Basic Notations:} 
For a set of vectors $\{v_i\}_{i \in S}$, $[v_i]_{i \in S}$ or $[v_1; \cdots; v_{|S|}]$ denotes their vertical stack. 
For a vector $v$ and a positive integer $i$, $[v]_i$ denotes the $i$-th entry of $v$. 
Denote $\crset{}{} \coloneqq \rset{}{} \cup \{+\infty\}$, $\rset{}{+} \coloneqq [0, +\infty)$, and $\rset{}{++} \coloneqq (0, +\infty)$. 
\Tblue{We let $\langle,\rangle$ represent the usual dot product, $\norm{\cdot}$ a general norm, and $\norm{\cdot}_*$ its dual.}

\section{PRELIMINARIES}

\subsection{Problem setup}
In this work, we consider an $N$-player game $\mathcal{G}$ which consists of a (finite) set of players $\playerN \coloneqq \{1, \ldots, N\}$. 
Each player has a continuous individual strategy space $\mathcal{X}_i \subseteq \rset{n_i}{}$, and by aggregating over all players, we write $\mathcal{X} \coloneqq \prod_{i \in \playerN} \mathcal{X}_i$ for the strategy space of the whole game. 
The action taken by the $i$-th player is denoted by a vector $x^i \in \mathcal{X}_i$, while the stack of the action vectors of other players is denoted by $x^{-i} = [x^j]_{j \in \mathcal{N}_{-i}} \in \mathcal{X}_{-i} \subseteq \rset{n_{-i}}{}$, 
where 
$\mathcal{N}_{-i} \coloneqq \playerN \backslash \{i\}$, 
$\mathcal{X}_{-i} \coloneqq \prod_{j \in \mathcal{N}_{-i}} \mathcal{X}_j$, 
and $n_{-i} \coloneqq \sum_{j \in \mathcal{N}_{-i}}n_j$. 
Each player $i$ is associated with an objective function $J_i(x^i; x^{-i})$, which is parameterized by the action $x^{-i}$ taken by others and determines player $i$ preferences for one action over the other. 
In other words, each player $i$ aims to optimize its local objective, given the actions taken by other players:
\begin{align}\label{eq:game-setup}
\minimize_{x^i \in \mathcal{X}_i} J_i(x^i; x^{-i}). 
\end{align}

\begin{assumption}\label{asp:objt}
For each player $i \in \playerN$, the objective $J_i$ is continuously differentiable in $x_i \in \mathcal{X}_i$ and $x_{-i} \in \mathcal{X}_{-i}$. 
\end{assumption}

\begin{assumption}\label{asp:fesb-set}
The local \Tblue{strategy space} $\mathcal{X}_i$ of each player $i$ is nonempty, closed, and convex. 
\end{assumption}

\begin{aspprime}{2}\label{asp:fesb-set-bdd}
In addition to the conditions in Assumption~\ref{asp:fesb-set}, each local \Tblue{strategy space} $\mathcal{X}_i$ is also bounded.
\end{aspprime}

One of the most widely used solution concepts in game theory is Nash equilibrium (NE), i.e., action profiles that discourage unilateral deviations, the formal definition of which is given below:
\begin{envdef}
The action profile $x^* \in \mathcal{X}$ is a Nash equilibrium if for all $i \in \playerN$ and every deviation $x^i \in \mathcal{X}_i$,
\begin{align}
    J_i(x_*^i; x_*^{-i}) \leq J_i(x^i; x_*^{-i}). 
\end{align}
\end{envdef}

A commonly used operator in NE seeking problems is pseudogradient $F:\mathcal{X} \to \rset{n}{}$.
It is defined as the Cartesian product of the \Tblue{subgradient $\partial_i J^i = \partial_{x^i} J^i$} for each $i \in \playerN$ \Tblue{or the stack of the partial gradient $\nabla_i J^i = \nabla_{x^i}J^i$ under the smoothness imposed in Assumption~\ref{asp:objt}}, i.e.
\Tblue{
\begin{align}\label{eq:psd-grad}
 F: x \mapsto \prod _{i \in \mathcal{N}} [\partial_{x^i} J^i(x^i; x^{-i})] = [\nabla_{x^i} J^i(x^i; x^{-i})]_{i \in \mathcal{N}}.
\end{align}
}

In the sequel, we will restrict our attention to potential games, a subclass of games that has been studied extensively in the context of congestion, traffic networks, oligopolies, etc. \cite{cohen2018acceleration} 
The problem described in \eqref{eq:game-setup} is called a potential game if it admits a potential function $\Phi: \mathcal{X} \to \rset{}{}$ such that for all player $i \in \playerN$, 
$J_i(x^i; x^{-i}) - J_i(y^i; x^{-i}) = \Phi(x^i, x^{-i}) - \Phi(y^i, x^{-i})$, 
for all $x^{-i} \in \mathcal{X}_{-i}$ and $x^i, y^i \in \mathcal{X}_i$. 
By letting $y^i \to x^i$ in the above equation, \Tblue{we obtain that
$F(x) = \nabla \Phi(x), \forall x \in \mathcal{X}$.} 
As has been discussed in \cite[Thm.~1.3.1]{facchinei2003finite}, showing that there exists a potential function $\Phi$ for a game $\mathcal{G}$ is equivalent to proving the Jacobian matrix of the pseudogradient $F$ is symmetric for all $x \in \mathcal{X}$. 
To regularize our \Tblue{problem}, we focus on a subclass of potential games whose potential functions satisfy the following assumption. 

\begin{assumption}\label{asp:cvx-potfuc}
The game $\mathcal{G}$ considered is a potential game with a convex and $L$-smooth potential function $\Phi$.
\end{assumption}

\subsection{Mirror map and its properties}

In this subsection, we briefly introduce the mirror map, its motivation, the associated concepts, and the related properties. 
For an optimization problem, when the objective function and the constraint set are well-behaved in a Euclidean space, we can apply the projected gradient descent which leverages the $\ell_2$ norm as a distance metric to measure the various quantity of interest. 
Nevertheless, sometimes we have to deal with more general situations in some Banach spaces $\mathcal{B}$'s, which utilize a class of norms not derived from inner products. 
Although it no longer makes sense to do the gradient descent in the primal space, as suggested by \cite[Ch.~3]{nemirovskii1983problem}, we can map the point in the primal space $\mathcal{B}$ to the corresponding dual space $\mathcal{B}^*$, perform the gradient update in the dual space, and map the updated point back to the primal space. 
We refer the interested reader to \cite[Ch.~4]{bubeck2014theory} for a more detailed discussion. 
The mapping from the dual to primal space is called the mirror map, which takes the following form: 
\begin{align}\label{eq:mr-map}
\begin{split}
\nabla \psi^*(z) = \argmax_{x \in \mathcal{X}} \big\{\langle z, x\rangle - \psi(x)\big\}
\end{split}
\end{align}
where $\psi: \mathcal{X} \to \rset{}{}$ is a strongly-convex differentiable function, called a regularizer (or penalty function), which is usually carefully chosen to suit the geometry of the problem; 
$\psi^*$ is the convex conjugate of $\psi(\cdot)$, i.e., $\psi^*(z) = \max_{x \in \mathcal{X}}\big\{ \langle z, x\rangle - \psi(x)\big\}$; 
the expression of $\nabla \psi^*$ denotes the subgradient of $\psi^*$; and \eqref{eq:mr-map} is a corollary of Danskin's Theorem. 
For instance, when optimizing a Lipschitz-continuous and convex function $f$, the mirror descent iteration can be written as $x_{k+1} \in \nabla \psi^*(\nabla\psi(x_k) - \nabla f(x_k))$ \cite[Sec.~4.2]{bubeck2014theory}\cite{juditsky2022unifying}. 
Below we include a property of the mirror map, which we will repeatedly leverage in our proof of the main results. 
\begin{lemma}
If $\psi$ is differentiable and $\mu$-strongly convex w.r.t. the norm $\norm{\cdot}$, then $\nabla \psi^*$ is \Tblue{$(1/\mu)$-Lipschitz}, i.e., 
$\norm{\nabla \psi^*(x) - \nabla \psi^*(x^\prime)} \leq \frac{1}{\mu} \norm{x - x^\prime}_*, \forall x, x^\prime \in \mathcal{B}^*$.
\end{lemma}
\begin{proof}
It directly follows from \cite[Prop.~14.2]{BauschkeHeinzH2017CAaM}. 
\end{proof}
\Tblue{The Bregman divergence associated with $\psi$ is defined as: }
\begin{align}
    D_{\psi}(x, x^\prime) \coloneqq \psi(x) - \psi(x^\prime) - \langle \nabla \psi(x^\prime), x - x^\prime\rangle,
\end{align}
for all $x, x^\prime \in \mathcal{X}$. 
One typical non-Euclidean example is \Tblue{the probability simplex}, where the objective function is defined on the simplex $\Delta_n \coloneqq \{x \in \rset{n}{+}: \sum_{i=1}^n [x]_i = 1\}$ \cite[Ch.~4.3]{bubeck2014theory}. 
A proper regularizer would be the negative entropy $\psi(x) \coloneqq \sum_{i=1}^n [x]_i \log([x]_i)$, which is $1$-strongly convex w.r.t. $\ell_1$ norm on $\Delta_n$ but not smooth (its gradient is not Lipschitz continuous at the boundary). 
The resulting mirror map admits the following analytical formulation: 
$\nabla \psi^*(x) = \frac{1}{\sum_{i=1}^n \exp([x]_i)} \cdot [\exp([x]_i)]_{i \in \{1, \ldots, n\}}$.






\section{NASH EQUILIBRIUM SEEKING WITH INSTANTANEOUS FEEDBACKS}

We begin with the blanket assumption that each player has instantaneous access to the perfect estimate of its partial gradient queried at the actions that all players choose. 
For a centralized optimization problem, \cite{cohen2018acceleration} proposes an improved accelerated gradient descent, named AGD+, which generalizes Nesterov's AGD \cite{Nesterov1983method}\cite[Sec.~3.6]{bubeck2014theory} and converges at the rate of $O(1/k^2)$. 
Besides, the convergence behaviors of AGD+ are analyzed in \cite{cohen2018acceleration} when the feedback is generated by some noisy and inexact first-order oracles. 
Here, we assume each player $i \in \playerN$ is endowed with a $\mu_i$-strongly convex regularizer $\psi_i: \mathcal{X}_i \to \rset{}{}$ and agrees on the step sizes for iterations. 
Let $\mu_* \coloneqq \min\{\mu_i: i \in \playerN\}$.
Based on AGD+, we propose Algorithm~\ref{alg:updt-no-delay}, \Tblue{which enables the group action profile's potential function value to asymptotically reach the minimum possible level}. 
We denote the action taken by player $i$ at the $k$-th iteration and the two associated auxiliary variables as $x^i_k$, $y^i_k$, and $z^i_k$. 
For brevity, we also let $x_k \coloneqq [x^i_k]_{i \in \playerN}$; similarly for $y_k$ and $z_k$. 
The positive step size sequences are denoted by $(a_k)_{k \in \nset{}{}}$ and we let $A_k \coloneqq \sum_{t=1}^k a_t$. 

\begin{algorithm}
\SetAlgoLined
\caption{Nash Equilibrium Seeking with Instantaneous Feedback (Player $i$)}
\label{alg:updt-no-delay}
\textbf{Initialize:} $x^i_{0} = x^i_{1} \in \mathcal{X}_i$ arbitrarily, $z^i_{0} = \nabla \psi^i(x^i_{0})$, $y^i_{0} = 0$, $g^i_{1} = \nabla_i J_i(x^i_{1};x^{-i}_{1})$, $A_0 = 0, a_k > 0$, and $A_k = \sum_{t=1}^k a_t$ for $k \in \nset{}{+}$\;
\textbf{At the $k$-th iteration ($k \in \nset{}{+}$)}: \\
\qquad $z^i_{k} \leftarrow z^i_{k-1} - a_kg^i_{k}$\;
\qquad $y^i_{k} \leftarrow \frac{A_{k-1}}{A_k}y^i_{k-1} + \frac{a_k}{A_k}\nabla \psi^*_i(z^i_{k})$\;
\qquad $x^i_{k+1} \leftarrow \frac{A_k}{A_{k+1}}y^i_{k} + \frac{a_{k+1}}{A_{k+1}}\nabla \psi^*_i(z^i_{k})$\;
\qquad Take action $x^i_{k+1}$ and receive the feedback $g^i_{k+1} = \nabla_i J_i(x^i_{k+1}; x^{-i}_{k+1})$ from the first-order oracle\;
\textbf{Return:} $\{y^i_{k}\}_{i \in \playerN}$.
\end{algorithm}

\begin{theorem}\label{thm:convg-rt-no-delay}
Consider an NE seeking problem in \eqref{eq:game-setup} with the players of the game following Algorithm~\ref{alg:updt-no-delay}. 
Suppose that Assumptions~\ref{asp:objt}, \ref{asp:fesb-set}, and \ref{asp:cvx-potfuc} hold, and $a_k$ is properly chosen such that $(a_k)^2/A_k \leq \mu_*/L$. 
Then at each iteration $k$ ($k \in \nset{}{+}$), 
\begin{align}
    \Phi(y_k) - \Phi(x_*) \leq \frac{D_{\psi}(x_*, x_0)}{A_k},
\end{align}
where $x_*$ is an arbitrary minimizer of the potential $\Phi$. 
\end{theorem}
\begin{proof}
See Appendix~\ref{pf:convg-rt-no-delay}. 
\end{proof}


\section{NASH EQUILIBRIUM SEEKING WITH FEEDBACK DELAYS}\label{sec:alg-delay}

\subsection{The general framework of delayed rewards}

In this section, we consider the environments where the feedback cannot be immediately received due to inherent delays, communication latency, etc. 
At the $k$-th stage of the process, each player $i$ chooses an action $x^i_k$, and $x^i_k$ together with the actions taken by other players are then observed by a perfect first-order oracle, which generates the exact gradient feedback $g^i_k$, by our standing assumption. 
However, unlike the setup in the previous section, a delay $d^i_k \geq 0$ is triggered and player $i$ cannot receive $g^i_k$ until the $(k + d^i_k)$-th stage. 
As a result, the feedback information available for players to update their next action may be obsolete with respect to the current state. 
In addition, it is worth emphasizing that $d^i_k$ \Tblue{varies} across different player $i$ at the iteration stage $k$ and can grow unbounded. 
To regularize the feedback delays that these players suffer from, we make the following assumption concerning the upper bounds of delays. 
\begin{assumption}\label{asp:delay}
For all players $i \in \playerN$, the delay $d^i_{k}$ for the first order feedback $g^i_{k}$ ($k \in \nset{}{+}$) satisfies $d^i_{k} \leq D \cdot k^\alpha$ for some $\alpha \geq 0$ and $D > 0$. 
\end{assumption}

Due to the flexibility of $d^i_k$, the arrival time of feedback is irregular, i.e., for some stages, player $i$ \Tblue{receives no feedback from} the past actions, while for some \Tblue{other} stages, several feedback messages originating from different earlier stages arrive at the same time. 
Heliou et al. \cite{heliou20agradient} let each player be endowed with a feedback 
(priority) queue and \Tblue{the received feedback is dequeued} in ascending priority order. 
To be more specific, at each iteration, player $i$ will append the feedback received to the queue and then use the feedback \Tblue{with the earliest timestamp} to do the updates. 
If the feedback queue is empty at some stages, this player will keep its action unchanged before it receives new feedback from the oracle.

\subsection{Nash equilibrium seeking with reward delays}

A variant of Algorithm~\ref{alg:updt-no-delay} is designed to \Tblue{contend with} the feedback delays of players by properly tuning the step sizes, as illustrated in Algorithm~\ref{alg:updt-delay}. 
The delayed feedback utilization scheme leveraged in Algorithm~\ref{alg:updt-delay} differs from that in \cite{heliou20agradient} in the sense that at each iteration, players update their action using the available first-order feedback with the highest priority repeatedly and neglect all the other feedback with low priority. 
Another difference is that, under this scheme, when no feedback is received for some iterations, player $i$ will apply the most recent first-order feedback message (in terms of the stage that it originates from) in history to update its auxiliary variables and action. 
\Tblue{For each player $i$, let $s^i(k) \in \nset{}{+}$ denote 
 the iteration from which the feedback used to update the $k$-th stage originates, and by $g^i_{s^i(k)} \coloneqq \nabla_{x^i} J_i(x^i_{s^i(k)}; x^{-i}_{s^i(k)})$ the corresponding first-order feedback. }
We have the following lemma to characterize the property of $s^i(k)$ and show that the gap between $s^i(k)$ and $k$ is upper bounded by some controllable length. 
\begin{lemma}\label{le:delay-bd}
If for each player $i$ at the $k$-th iteration ($\forall k \in \nset{}{+}$), the delay $d^i_k$ satisfies $d^i_k \leq Dk^\alpha$ for some $\alpha \geq 0$ and $D > 0$, then $s^i(k) + 1 + D(s^i(k) + 1)^\alpha > k, \forall k \in \nset{}{+}$. 
\end{lemma}
\begin{proof}
Assume ad absurdum that $s^i(k) + 1 + D(s^i(k) + 1)^{\alpha} \leq k$. 
Then $s^i(k) + 1 + d^i_{s^i(k) + 1} \leq k$ and the feedback message generated from the actions taken at the $(s^i(k) + 1)$-th stage would arrive before or at the $k$-th iteration. 
The first-order information $g^i_{s^i(k)}$ would be discarded at an earlier stage or would not be accepted at all with the more recent feedback message $g^i_{s^i(k)+1}$ in the cache, which generates an immediate contradiction. 
\end{proof}

By carefully controlling the interplay between delays and step sizes, we can establish the relationship between the convergence rates that Algorithm~\ref{alg:updt-delay} can achieve and the powers of the upper bounds for delays as follows. 

\begin{theorem}\label{thm:convg-rt-delay}
Consider an NE seeking problem in \eqref{eq:game-setup} with the players of the game following Algorithm~\ref{alg:updt-delay}. 
Suppose that Assumptions~\ref{asp:objt}, \ref{asp:fesb-set-bdd}, and \ref{asp:cvx-potfuc} hold, the delays affecting the algorithm are under Assumption~\ref{asp:delay} with $\alpha < 1$, and $a_k = a_0\cdot k^\beta \;(0 \leq \beta \leq 1)$ is properly chosen such that $(a_k)^2/A_k \leq \mu_*/L$. 
Then, 
\begin{align}
    \Phi(y_k) - \Phi(x_*) = O(\frac{1}{k^{1 - \alpha}}),
\end{align}
where $x_*$ is an arbitrary minimizer of the potential $\Phi$. 
\end{theorem}
\begin{proof}
See Appendix~\ref{pf:thm:convg-rt-delay}
\end{proof}
\Tblue{Expanding on Theorem~\ref{thm:convg-rt-delay}, we now delve into the cases where the delay upper bound grows at a linear or superlinear rate and examine the conditions that guarantee convergence. }
\begin{coro}\label{coro:convg-rt-lnr-delay}
Consider an NE seeking problem in \eqref{eq:game-setup} with the players of the game following Algorithm~\ref{alg:updt-delay}. 
Suppose that Assumptions~\ref{asp:objt}, \ref{asp:fesb-set-bdd}, and \ref{asp:cvx-potfuc} hold, the delays affecting the algorithm are under Assumption~\ref{asp:delay} with $\alpha = 1$, and $a_k = a_0\cdot 1/k$ for some $a_0 > 0$. 
Then,
\begin{align}
    \Phi(y_k) - \Phi(x_*) = O(\frac{\log\log(k)}{\log(k)}),
\end{align}
where $x_*$ is an arbitrary minimizer of the potential $\Phi$. 
\end{coro}
\begin{proof}
See Appendix~\ref{pf:thm:convg-rt-spl-delay}
\end{proof}

\begin{coro}\label{coro:convg-rt-sup-delay}
Consider an NE seeking problem in \eqref{eq:game-setup} and the players of the game follow Algorithm~\ref{alg:updt-delay}. 
Suppose that Assumptions~\ref{asp:objt}, \ref{asp:fesb-set-bdd}, and \ref{asp:cvx-potfuc} hold, the delays affecting the algorithm are under Assumption~\ref{asp:delay} with $\alpha > 1$.
The sequence $(a_k)_{k \in \nset{}{+}}$ satisfies $a_k = a_0/((k+1)\log (k+1))$ for some $a_0 > 0$. 
Then,
\begin{align}
    \Phi(y_k) - \Phi(x_*) = O(\frac{\log\log\log(k)}{\log\log(k)}),
\end{align}
where $x_*$ is an arbitrary minimizer of the potential $\Phi$. 
\end{coro}
\begin{proof}
See Appendix~\ref{pf:thm:convg-rt-spl-delay}
\end{proof}

\begin{algorithm}
\SetAlgoLined
\caption{Nash Equilibrium Seeking with Delayed Feedback (Player $i$)}
\label{alg:updt-delay}
\textbf{Initialize:} $x^i_{0} = x^i_{1} \in \mathcal{X}_i$ arbitrarily, $z^i_{0} = \nabla \psi^i(x^i_{0})$, $y^i_{0} = 0$, $g^i_\star = g^i_{1} = \nabla_i J_i(x^i_{1};x^{-i}_{1})$, $s^i(0) = 1$, 
$A_0 = 0, a_k > 0$, and $A_k = \sum_{t=1}^k a_t$ for $k \in \nset{}{+}$\;
\textbf{At the $k$-th iteration ($k \in \nset{}{+}$)}: \\
Receive $\mathcal{R}^i_k \coloneqq \{(t, g^i_t): k-1 < t + d^i_t \leq k\}$\;
$s^i(k) \leftarrow s^i(k-1)$\;
\If{$\mathcal{R}^i_k \neq \varnothing$}{
Obtain the most recent pair $(t_{\max}, g^i_{t_{\max}})$\;
\If{$t_{\max} > s^i(k)$}{
$s^i(k) \leftarrow t_{\max}$, $g^i_\star \leftarrow g^i_{s^i(k)}$\;
}
}
$z^i_k \leftarrow z^i_{k-1} - a_kg^i_\star$ \;
$y^i_{k} \leftarrow \frac{A_{k-1}}{A_k}y^i_{k-1} + \frac{a_k}{A_k}\nabla \psi^*_i(z^i_{k})$\;
$x^i_{k+1} \leftarrow \frac{A_k}{A_{k+1}}y^i_{k} + \frac{a_{k+1}}{A_{k+1}}\nabla \psi^*_i(z^i_{k})$\;
Take action $x^i_{k+1}$; the feedback message $(k+1, g^i_{k+1})$ with $g^i_{k+1} \coloneqq \nabla_i J_i(x^i_{k+1}; x^{-i}_{k+1})$ will arrive at the $(k+1+d^i_{k+1})$-th iteration\;
\textbf{Return:} $\{y^i_{k}\}_{i \in \playerN}$.
\end{algorithm}

\section{CASE STUDY AND NUMERICAL SIMULATIONS}
\subsection{Problem setup of routing games}
Due to the significant economic costs that it causes, traffic congestion is a well-recognized issue, especially in major metropolitan centers \cite{paccagnan2018nash}. 
In this study, we evaluate the performance of the proposed algorithm with a class of routing games.
In this game, a directed graph $(\mathcal{V}, \mathcal{E})$ is given, with the node set $\mathcal{V} \coloneqq \{1, \ldots, V\}$ denoting geographical locations and the directed edge set $\mathcal{E} \coloneqq \{1, \ldots, E\} \subseteq \mathcal{V} \times \mathcal{V}$ denoting the available roads connecting different locations.  
Each player $i \in \playerN$ represents a driver or a group of drivers who share a pre-determined original-destination (O/D) pair $(o^i, d^i) \in \mathcal{V} \times \mathcal{V}$, and it needs to determine a route to minimize the travel time, which is at the same time influenced by the routes chosen by other players. 

Same as the basic setup in \cite{vu2021fast}, the routing game considered here consists of the following three essential components: the structure of the road network, the congestion costs, and the feasible flow profiles. 
For the part of the network structure, in addition to the setup given above, for each player $i$, there exist a set of feasible routes joining $o^i$ and $d^i$, which is denoted by $\mathcal{P}^i$. 
The cardinality of $\mathcal{P}^i$ is represented by $P^i$. 
Given the road network, the strategy space of player $i$ can be formally described by a scaled simplex $\mathcal{X}^i \coloneqq \{x^i \in \rset{P^i}{+}: \sum_{p=1}^{P^i} [x^i]_p = S_i\}$, with $S_i$ representing the traffic demand of player $i$ and each entry $[x^i]_p$ representing the scaled probability of choosing route $p$. 
Finally, we use $J^i_p:\mathcal{X}^i \to \rset{}{+}$ to denote the congestion cost functions that measure latency associated with route $p \in \mathcal{P}^i$ for each player $i \in \playerN$. 
The formal expression of $J^i_p$ can be written as $J^i_p(x^i; x^{-i}) = \sum_{e \in p} J_e(\ell_e(x^i, x^{-i}))$. 
In the above, $\ell_e(x) = \sum_{i \in \playerN}\sum_{p \in \mathcal{P}^i} \mathds{1}_{\{e \in p\}}[x^i]_p$ reflects the traffic load on edge/road $e$ under the collective action profile $x \coloneqq [x^i]_{i \in \playerN}$. 
The road impedance function or edge cost function $J_e$ for each $e \in \mathcal{E}$ relates the travel time to the traffic load and the road condition, the typical examples of which include the Bureau of Public Roads (BPR) function $J_e(\ell) = a_e(1 + b_e(\ell / c_e)^{r_e})$ with the road condition characterized by the parameters $a_e$, $b_e$, $c_e$, and $r_e$.
We refer the interested reader to \cite[Sec.~2]{vu2021fast} and references therein for further details and other conventional examples of $J_e$. 
\Tblue{Instead of considering the NE as before, we now focus on the Wardrop equilibrium, which serves as an alternative solution concept. }
\begin{envdef}
(Wardrop equilibrium) The action profile $x_* \in \mathcal{X}$ is a Wardrop equilibrium if for all $i \in \playerN$ and arbitrary $p, q \in \mathcal{P}^i$, 
$ [x^i_*]_p > 0 \implies J^i_p(x^i_*; x^{-i}_*) \leq J^i_q(x^i_*; x^{-i}_*)$ . 
\end{envdef}

See \cite[Def.~2]{paccagnan2018nash}\cite[Sec.~22.2.2 \& Sec.~22.2.6]{NisaRougTardVazi07} for the definitions of Wardrop equilibria under different problem setups. 
The motivation to consider Wardrop equilibria rather than Nash equilibria here is that the former admits a potential function $\Phi$ such that the minimizers of $\Phi$ coincide with the Wardrop equilibria of the congestion game. 
The potential function $\Phi$, also known as the Beckmann-McGuire–Winsten (BMW) potential \cite[Sec.~3.1]{beckmann1956studies}, is of the form
$\Phi(x) \coloneqq \sum_{e \in \mathcal{E}} \int^{\ell_e(x)}_0 J_e(\tau)d\tau, \;\text{where}\;x \in \prod_{i \in \playerN} \mathcal{X}^i$. 
In the deterministic setting, it directly follows from the Leibniz integral rule that the gradient is given by 
$\frac{\partial \Phi(x)}{\partial [x^i]_p} = \sum_{e\in \mathcal{E}} J_e(\ell_e(x))\cdot \frac{\partial \ell_e(x)}{\partial [x^i]_p} 
 = \sum_{e\in \mathcal{E}} J_e(\ell_e(x))  \cdot \mathds{1}_{\{e \in p\}} = J^i_p(x)$. 

The learning routine for this congestion game can be briefly summarized as follows: at each iteration, this group of players follow the routes from the origins to destinations suggested by the last iteration, run the experiments, receive the travel times that may originate from some previous iterations, and update the action profiles. 

\subsection{Experimental details and simulation results}
We perform a numerical analysis based on the dataset of the eastern Massachusetts highway, which consists of $74$ nodes and $258$ directed edges \cite{TransportationNetworks}. 
There are in total $200$ players (O/D pairs) and each player has $20$ different routes to choose from. 
The BPR function is leveraged to map the traffic load on each edge to the commensurate travel time. 
For each road $e \in \mathcal{E}$, the free-flow parameter $a_e$ is randomly sampled from the uniform distribution $U[2, 3]$. Similarly, the coefficient $b_e$, capacity $c_e$, and the power $r_e$ are sampled from $U[3, 13]$, $U[60, 80]$, and $U[1, 1.5]$, respectively. 
The traffic demand for each player $i \in \playerN$ is randomly sampled from $U[10, 20]$. 
With the above parameters chosen, the potential function $\Phi$ is convex and $L$-smooth on the direct product of scaled simplices. 
The scaled negative entropy is leveraged as the regularizer for the proposed algorithms, i.e., for each player $i \in \playerN$, $\psi_i(x) \coloneqq \sum_{p=1}^{P_i} \frac{[x^i]_p}{S_i} \log(\frac{[x^i]_p}{S_i})$, and its mirror map can be derived as 
$\nabla \psi^*_i(x) \coloneqq \Big[\frac{S_i \exp([x^i]_p)}{\sum_{q = 1}^{P^i} \exp([x^i]_q)}\Big]_{p = 1, \ldots, P_i}$. 

In the first experiment, we let this group of players be affected by the same deterministic delay function $d^i_k$, with the results illustrated in Fig.~\ref{fig:deter-delay}. 
For Case~1, all the players have access to immediate feedback and can update their strategy instantaneously. 
As reflected in Fig.~\ref{fig:deter-delay}, the metric $\Delta \Phi_k \coloneqq \Phi(y_k) - \Phi(x_*)$ decays at a rate of $O(1/k^2)$, which matches the result of Theorem~\ref{thm:convg-rt-no-delay}. 
For Cases~2 and 3, a constant delay time is enforced at each iteration, and the convergence rate degrades from $O(1/k^2)$ to $O(1/k)$. 
In addition, for Case~3 which suffers from a larger constant delay, the convergence rate at the first $10^2$ iterations is slower than that of Case~2, while after the $10^2$-th iteration, the rate matches that of Case~2. 
As the power of the delay function increases but remains less than 1, the convergence rate of the proposed algorithm decreases commensurately in Cases~4 and 5. 
Lastly, we apply the proposed algorithm to Case~6 with linear delay, the metric $\Delta \Phi_k$ keeps a descending trend, nevertheless at a relatively conservative rate. 
\Tblue{The curves of Case~2 to 5 exhibit oscillations since when the expected feedback is delayed, players use the most recent feedback available as a substitute, which introduces some approximation error. }

To better model the realistic setup where different players suffer from various delays and update asynchronously, we conduct the second experiment, where for each player $i$, the feedback $g^i_k$ arrives at the $\lceil k +D^i_k \rceil$-th iteration, and \Tblue{$D^i_k$ is assumed to be a uniform random variable with the distribution $U[0, 2d^i_k]$ such that $\expt{}{D^i_k} = d^i_k$ with $d^i_k$ defined in the first experiment}. 
The results are included in Fig.~\ref{fig:stoch-delay}. 
We observe empirically that except for Case~2 where only a small constant delay plays its role, the convergence speed improves when the feedback is received and the update is conducted asynchronously, as demonstrated by Cases~3 to 6. 

\begin{figure}
    \centering
    \includegraphics[width=0.45\textwidth]{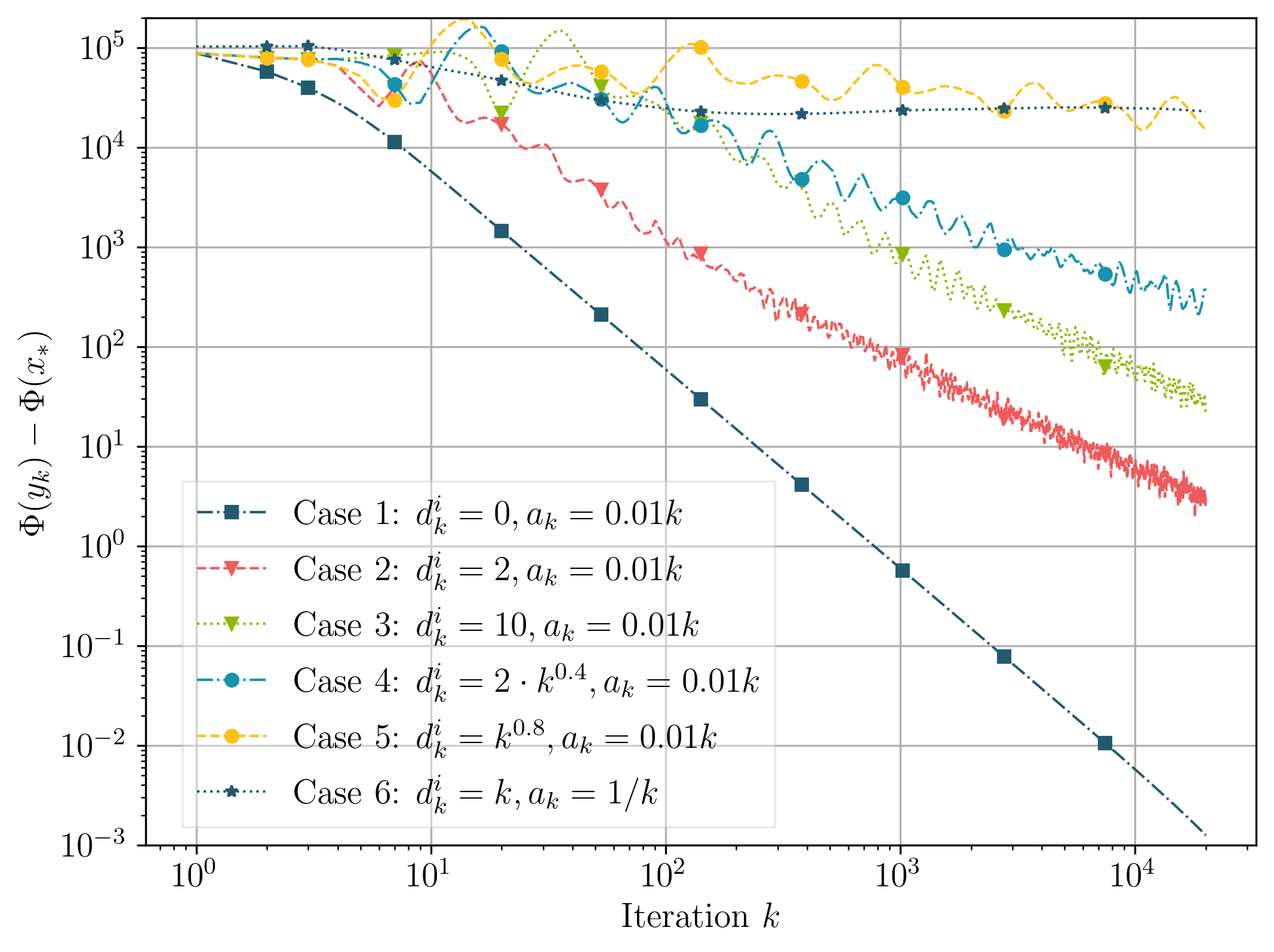}
    \caption{Convergence Rates of Algorithms~\ref{alg:updt-no-delay} \& \ref{alg:updt-delay} Subject to Different Types of Deterministic Delays}
    \label{fig:deter-delay}
\end{figure}

\begin{figure}
    \centering
    \includegraphics[width=0.45\textwidth]{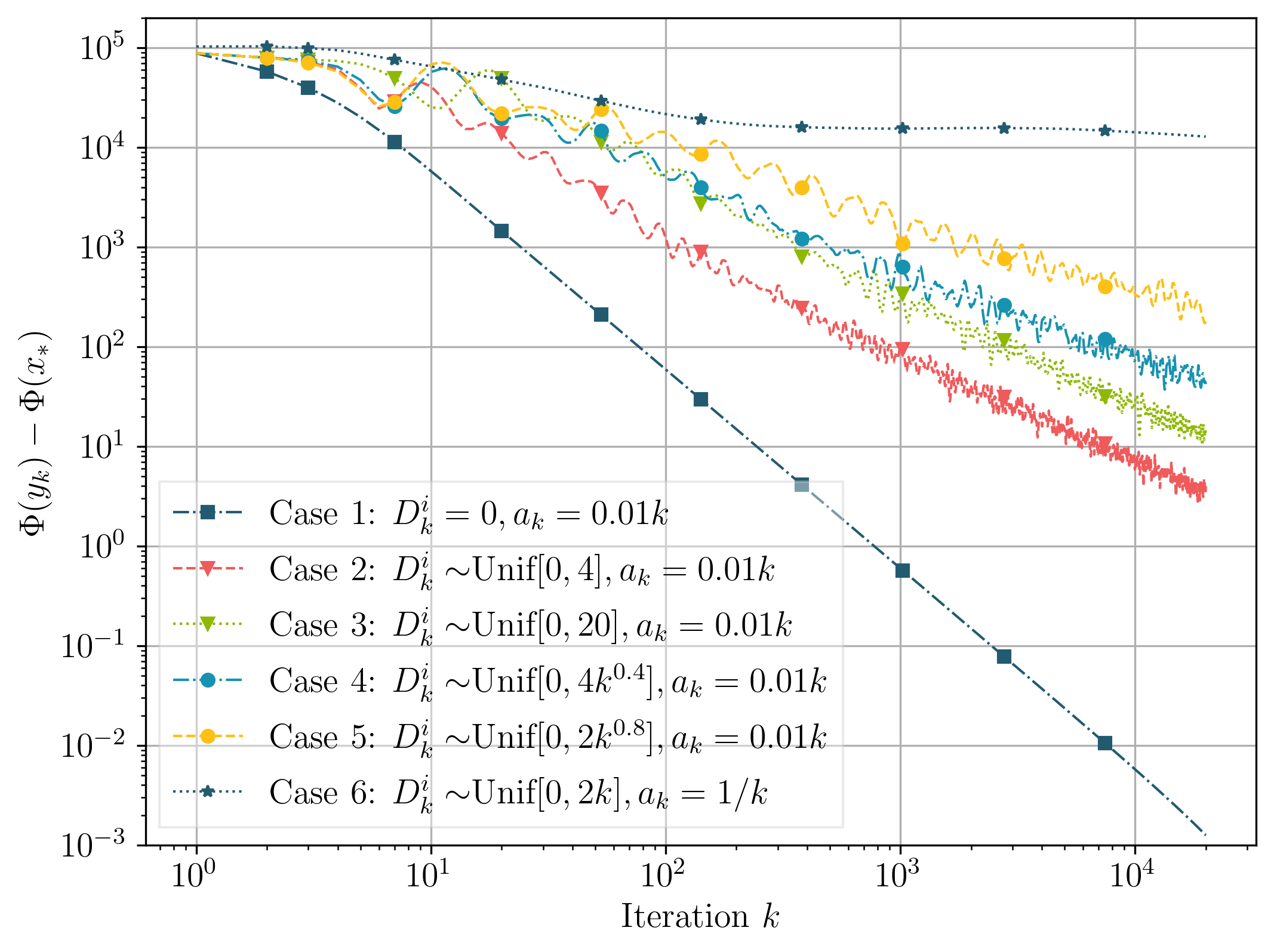}
    \caption{Convergence Rates of Algorithm~\ref{alg:updt-delay} Subject to Stochastic Delays with Different Upper Bounds}
    \label{fig:stoch-delay}
\end{figure}

\section{CONCLUSIONS}

This paper studies the convergence properties of a proposed algorithm for multi-player potential games with feedback delays. 
For games with instantaneous feedback, it directly follows from the characteristics of the centralized version of the proposed algorithm that it will converge at a rate of $O(1/k^2)$. 
When delays exist and are upper-bounded by some \Tblue{fraction power} functions, convergence can be guaranteed by properly tuning the step sizes, while the convergence rates are compromised accordingly. 
There remain several open problems. 
The first one concerns the regret analysis of the algorithm under the accelerated scheme considered since in the information parsimonious environment, a no-regret algorithm would be the players' most sensible choice. 
In addition, we note that even though the algorithms work in a decentralized manner and each player updates independently, all players should agree to use the same sequence of step sizes. 
Exploring the possibility to use heterogeneous or even adaptive step sizes is one of our future directions.





\bibliographystyle{IEEEtran}
\bibliography{IEEEabrv,references}

\appendices
\input{appendix}


\end{document}

%% file: macros.tex
\newcommand{\norm}[1]{\lVert#1\rVert}

\newcommand{\expt}[2]{\mathbb{E}_{#1}[#2]}

\DeclareMathOperator{\minimize}{minimize}

\newcommand{\playerN}{\mathcal{N}}

\newcommand{\rset}[2]{\mathbb{R}^{#1}_{#2}}

\newcommand{\crset}[2]{\overline{\mathbb{R}}^{#1}_{#2}}
\newcommand{\nset}[2]{\mathbb{N}^{#1}_{#2}}

\DeclareMathOperator{\argmax}{argmax}

\usepackage{accents}

\newcommand{\Tblue}[1]{\textcolor{black}{#1}}

%% file: appendix.tex
\section*{Appendix}
\renewcommand{\thesubsection}{\Alph{subsection}}

\newtheorem{appdxlemma}{Lemma}
\numberwithin{appdxlemma}{subsection} 

\subsection{Approximate Duality Gap and Technical Lemmas}\label{appd:approx-dual-gap}
For the completeness of the manuscript, in this subsection of Appendix, we will briefly summarize the main results in \cite{cohen2018acceleration}, upon which the convergence rate results in our main text are established. 
To construct an accelerated gradient descent which allows for an inexact oracle, the authors of \cite{cohen2018acceleration} leverage an approximate duality gap for each stage $k$, with its upper bound denoted by $U_k$ and lower bound $L_k$. 
Three points $(x_k, y_k, z_k)$ are maintained throughout the iterations: the gradient is queried at $x_k$, $y_k$ serves as the primal solution at the end of each iteration $k$, and $z_k$ the dual solution. 
The bounds $U_k$ and $L_k$ are designed such that for the function $f$ that we would like to minimize, $f(y_k) - f(x_*) \leq G_k = U_k - L_k$, where $x_*$ represents an arbitrary point in the solution set. 
Let $(a_k)_{k \in \nset{}{}}$ be a positive sequence and $A_k = \sum_{t=1}^{k} a_t$. 
The inexact gradient information is expressed as $\Tilde{\nabla}f(y_k) = \nabla f(y_k) + \eta_k$. 
The authors of \cite{cohen2018acceleration} select $U_k$ and $L_k$ as follows:
\begin{align*}
& U_k = f(y_k) \\
& L_k = \frac{1}{A_k}\Big(\sum_{t=1}^{k}a_tf(x_t) - \sum_{t=1}^{k}a_t\langle \eta_t, x_* - x_t\rangle - D_{\psi}(x_*, x_0) \\
& \qquad + \min_{u \in \mathcal{X}}\big\{\sum_{t=1}^{k}a_t\langle \Tilde{\nabla}f(x_t), u - x_t\rangle + D_{\psi}(u, x_0)\big\}\Big).
\end{align*}
The essential idea is that if the difference $E_k = A_kG_k - A_{k-1}G_{k-1}$ can be upper bounded, then by the telescoping sum $G_k = \frac{A_1}{A_k}G_k + \frac{1}{A_k}\sum_{t=2}^{k}E_t$, one can obtain some insights regarding the decaying rate of $f(y_k) - f(x_*)$. 
Tailored to the duality gap given above, the steps of (AGD+) are defined as: 
\begin{align*}
& x_k \leftarrow \frac{A_{k-1}}{A_k}y_{k-1} + \frac{a_k}{A_k} \nabla \psi^*(z_{k-1}), \\ 
& z_k \leftarrow z_{k-1} - a_k\Tilde{\nabla}f(x_k), \qquad\qquad\qquad\qquad \text{(AGD+)} \\
& y_k \leftarrow \frac{A_{k-1}}{A_k}y_{k-1} + \frac{a_k}{A_k}\nabla \psi^*(z_k), 
\end{align*}
with $x_1 = x_0$, $y_1 = v_1 = \nabla \psi^*(z_1)$, and $v_k \coloneqq \nabla \psi^*(z_k)$. 

An intermediate result that works as a heavy-lifting tool for our later convergence rate analysis is provided in the following lemma about the separation of $E_k$. 
\begin{appdxlemma}\label{le:cohen}
(\cite[Lemma~3.2]{cohen2018acceleration}) 
Let $E^\eta_k = a_k\langle \eta_k, x_* - v_k\rangle$ and $E^e_k = A_k(f(y_k) - f(x_k)) - A_k\langle \nabla f(x_k), y_k - x_k\rangle - D_{\psi}(v_k, v_{k-1})$. 
Then $E_k \leq E^\eta_k + E^e_k$. 
\end{appdxlemma}

\subsection{Proof of Theorem \ref{thm:convg-rt-no-delay}}\label{pf:convg-rt-no-delay}

We begin by considering the compact formulations for the updates in Algorithm~\ref{alg:updt-no-delay}. 
Let $\psi(z) \coloneqq \sum_{i \in \playerN} \psi_i(z^i)$, and it directly follows that \Tblue{$\nabla \psi^*(z) = \big[\argmax_{x^i \in \mathcal{X}_i}\{\langle z^i, x^i\rangle - \psi_i(x^i)\}\big]_{i \in \playerN} = \big[\nabla \psi^*_i(z^i)\big]_{i \in \playerN}$}. 
Moreover, since we assume that $\mathcal{G}$ admits a potential function $\Phi$, the pseudogradient at point $x$ equals the gradient of $\Phi$ at this point, i.e., $F(x) = \nabla \Phi(x)$. 
Then Algorithm~\ref{alg:updt-no-delay} can be compactly written as follows: 
\begin{align*}
& z_k \leftarrow z_{k-1} - a_k\nabla \Phi(x_k) \\
& y_k \leftarrow \frac{A_{k-1}}{A_k} y_{k-1} + \frac{a_k}{A_k} \nabla \psi^*(z_k) \\
& x_{k+1} \leftarrow \frac{A_{k}}{A_{k+1}} y_k + \frac{a_{k+1}}{A_{k+1}} \nabla \psi^*(z_k).
\end{align*}
Then we can proceed similar to the proof of \cite[Thm.~3.4]{cohen2018acceleration}. 
Nevertheless, we include detailed steps for the completeness of this paper. 
Let $v_k \coloneqq \nabla \psi^*(z_k)$. 
Based on the results in Appendix~\ref{appd:approx-dual-gap} and using the same set of notations ($U_k, L_k, G_k$, etc.), we can bound the change of subsequent iterations $E_k$ in the following way: 
\begin{align*}
& E_k = A_kG_k - A_{k-1}G_{k-1} \overset{(a)}{\leq} A_k(\Phi(y_k) - \Phi(x_k))\\
& \quad - A_k\langle \nabla \Phi(x_k), y_k - x_k\rangle - D_{\psi}(v_k, v_{k-1}) \\
& \overset{(b)}{\leq} A_k\cdot\frac{L}{2}\norm{x_k - y_k}^2 - \frac{\mu_*}{2}\norm{v_k - v_{k-1}}^2 \\
& = (\frac{L(a_k)^2}{2A_k} - \frac{\mu_*}{2})\norm{v_k - v_{k-1}}^2 \overset{(c)}{\leq} 0, 
\end{align*}
where 
$(a)$ follows from Lemma~\ref{le:cohen} with the first-order error $\eta_k = 0$ for all $k \in \nset{}{+}$;
$(b)$ is the result of the $L$-smoothness of $\Phi$ and the fact that $D_\psi(x, x^\prime) \geq \mu_*/2\norm{x - x^\prime}^2$; 
$(c)$ is due to the condition that $(a_k)^2/A_k \leq \mu_*/L$. 
Hence, we can arrive at the following conclusion about the deterministic convergence rate:
\begin{align*}
& \Phi(y_k) - \Phi(x_*) \leq G_k \leq \frac{A_1}{A_k}G_1 \\
& \leq \frac{A_1}{A_k}\Big(\Phi(y_1) - \Phi(x_1) - \langle \nabla \Phi(x_1), v_1 - x_1\rangle \\
& \qquad + \frac{1}{A_1}(D_{\psi}(x_*, x_0) - D_{\psi}(v_1, x_0))\Big) \\
& \overset{(a)}{\leq} \frac{D_{\psi}(x_*, x_0)}{A_k} + \frac{A_1}{2A_k}(\frac{L(a_1)^2}{A_1} - \mu_*)\norm{v_1 - v_0}^2 \\
& \leq \frac{D_{\psi}(x_*, x_0)}{A_k},
\end{align*}
where $(a)$ follows from the initialization that $x_0 = x_1$ and $v_1 = y_1$.

\subsection{Convergence Rate Analysis with Feedback Delays}\label{pf:thm:convg-rt-delay}

In the next lemma, we analyze the difference between pseudogradients from two subsequent iterations, which will be later leveraged to control the error introduced by delays.  
\begin{appdxlemma}\label{le:grad-bd}
Under the same assumptions of Theorem~\ref{thm:convg-rt-delay}, for each player $i \in \playerN$, at the $\ell$-th iteration, the dual norm of the difference between subsequent partial gradients queried at $x_{\ell+1}$ and $x_{\ell}$ are bounded as follows:
\begin{align}
\norm{g^i_{\ell+1} - g^i_{\ell}}_* \leq LD_{\mathcal{X}}\frac{a_{\ell} + a_{\ell+1}}{A_{\ell+1}},
\end{align}
where $D_{\mathcal{X}}$ is the diameter of the feasible set $\mathcal{X}$. 
\end{appdxlemma}
\begin{proof}
Based on the updating steps given in Algorithm~\ref{alg:updt-delay}, we can expand the difference $x_{\ell+1}$ and $x_{\ell}$ in the following way:
\begin{align*}
& x_{\ell+1} = \frac{A_\ell}{A_{\ell+1}}y_{\ell} + \frac{a_{\ell+1}}{A_{\ell+1}}v_{\ell}, y_{\ell} = x_{\ell} - \frac{a_{\ell}}{A_{\ell}}v_{\ell-1} + \frac{a_{\ell}}{A_{\ell}}v_{\ell}, \\
& x_{\ell+1} = \frac{A_{\ell}}{A_{\ell+1}}x_{\ell} - \frac{a_{\ell}}{A_{\ell+1}}v_{\ell-1} + \frac{a_{\ell} + a_{\ell+1}}{A_{\ell+1}}v_{\ell}, \\
& x_{\ell+1} - x_{\ell} = \frac{a_{\ell}}{A_{\ell+1}}(v_{\ell} - v_{\ell-1}) + \frac{a_{\ell+1}}{A_{\ell+1}}(v_{\ell} - x_{\ell}). 
\end{align*}
Since $v_{\ell}$ and $v_{\ell-1}$ are the results of mirror maps, and $x_{\ell}$ and $x_{\ell+1}$ are the convex combinations of the points inside the feasible set $\mathcal{X}$, we have
\begin{align*}
\norm{x_{\ell+1} - x_{\ell}} &\leq \frac{a_{\ell}}{A_{\ell+1}}\norm{v_{\ell} - v_{\ell-1}} + \frac{a_{\ell+1}}{A_{\ell+1}}\norm{v_{\ell} - x_{\ell}} \\
& \leq \frac{a_\ell + a_{\ell+1}}{A_{\ell+1}}D_{\mathcal{X}}. 
\end{align*}
Finally, by the $L$-smoothness of the potential function $\Phi$, we can conclude $\norm{g^i_{\ell+1} - g^i_{\ell}}_* \leq LD_{\mathcal{X}}(a_{\ell} + a_{\ell+1})/A_{\ell+1}$.
\end{proof}

\begin{proof}
(Proof of Theorem~\ref{thm:convg-rt-delay}) 
By leveraging Lemma~\ref{le:cohen}, the difference between the potential value at the $k$-th iteration and the optimal potential value can be bounded as follows:
\begin{align}\label{eq:delay-upp-bd}
\begin{split}
& \Phi(y_k) - \Phi(x_*) \leq G_k \leq \frac{A_1}{A_k}G_1 + \frac{1}{A_k} \sum_{t=1}^{k} E_t\\
& \overset{(a)}{\leq} \Big(\frac{A_1}{A_k}G_1 + \frac{1}{A_k} \sum_{t=1}^{k} E^e_t \Big) + \frac{1}{A_k} \sum_{t=1}^{k} E^\eta_t \\
& \overset{(b)}{\leq} \frac{1}{A_k}D_{\psi}(x_*, x_0) + \frac{1}{A_k}\sum_{t=1}^{k}a_t \langle \eta_t, x_* - v_t\rangle,
\end{split}
\end{align}
where $(a)$ is the direct result of Lemma~\ref{le:cohen}; by using the same analysis in Theorem~\ref{thm:convg-rt-no-delay} and the definition of $E^\eta_t$, we then have the inequality in $(b)$. 
Here, the second part in the last line emerges as the result of the delays in first-order oracles and particularly $\eta_1 = 0$. 
We can further separate this part regarding the contribution of each player:
\begin{align*}
& \sum_{t=1}^{k}a_t \langle \eta_t, x_* - v_t\rangle 
= \sum_{i \in \playerN}\sum_{t=1}^{k} a_t\langle \eta^i_t, x^i_* - v^i_t\rangle, 
\end{align*}
where $\eta^i_t = (g^i_{t-1} - g^i_t) + (g^i_{t-2} - g^i_{t-1}) + \cdots + (g^i_{s^i(t)} - g^i_{s^i(t)+1})$. 
For each player $i$, we next proceed to establish an upper bound for the sum of inner products associated with $\eta^i_t$:
\begin{align*}
& \sum_{t=1}^{k} a_t\langle \eta^i_t, x^i_* - v^i_t\rangle \leq \sum_{t=1}^{k} a_t \norm{\eta^i_t}_* \cdot \norm{x^i_* - v^i_t} \\
& \leq D_{\mathcal{X}}\sum_{t=1}^{k} a_t \norm{\eta^i_t}_* \leq D_{\mathcal{X}}\sum_{t=1}^{k} a_t \sum_{\ell=s^i(t)+1}^{t} \norm{g^i_{\ell-1} - g^i_{\ell}}_* \\
& \leq L(D_{\mathcal{X}})^2\sum_{t=1}^{k}a_t \sum_{\ell=s^i(t)+1}^{t}\frac{a_\ell + a_{\ell-1}}{A_{\ell}} \\
& \leq 2L(D_{\mathcal{X}})^2\sum_{t=1}^{k}a_t \sum_{\ell=s^i(t)+1}^{t}\frac{a_\ell}{A_{\ell}} \\ 
& \overset{(a)}{\leq} 2L(D_{\mathcal{X}})^2(\beta + 1) \sum_{t=1}^{k} a_t \sum_{\ell = s^i(t)+1}^{t} \frac{1}{\ell} \\
& \leq 2L(D_{\mathcal{X}})^2(\beta + 1) \sum_{t=1}^{k} a_0 t^{\beta} \cdot \frac{t - s^i(t)}{s^i(t) + 1} \\
& \overset{(b)}{\leq} 2L(D_{\mathcal{X}})^2(\beta + 1)a_0 \sum_{t=1}^{k} t^{\beta} \cdot \Big(\frac{1}{s^i(t)+1} + \frac{D}{(s^i(t)+1)^{1 - \alpha}}\Big),
\end{align*}
where $(a)$ is due to the inequality that $A_{\ell} \geq \int^{\ell}_0 a_0\tau^\beta d\tau = a_0/(\beta+1) \ell^{\beta+1}$; 
$(b)$ follows from $t - s^i(t) \leq 1 + D(s^i(t)+1)^{\alpha}$ as suggested in Lemma~\ref{le:delay-bd}. 
For the above two summations in the last line with $0 \leq \alpha < 1$ and $0 < \beta \leq 1$, since $(D+1)(s^i(t) + 1) \geq s^i(t) + 1 + D(s^i(t) + 1)^{\alpha} \geq t$, we have
\begin{align*}
& \sum_{t=1}^{k} \frac{t^{\beta}}{s^i(t)+1} \leq \sum_{t=1}^{k}(D+1)\frac{t^{\beta}}{t} \\
& \qquad \leq (D+1)\int_{0}^k \tau^{\beta-1}d\tau \leq \frac{D+1}{\beta} k^{\beta} \\
& \sum_{t=1}^{k} D \frac{t^\beta}{(s^i(t)+1)^{1-\alpha}} 
\leq \sum_{t=1}^{k} D(D+1)^{1-\alpha} \frac{t^{\beta}}{t^{1-\alpha}} \\
& \qquad \overset{(a)}{\leq} D(D+1)^{1-\alpha} \Big(\int^{k}_{0} \tau^{\alpha + \beta - 1} d\tau + k^{\alpha + \beta - 1}\Big) \\
& \qquad = D(D+1)^{1-\alpha} \Big(\frac{1}{\alpha + \beta}k^{\alpha + \beta} + k^{\alpha + \beta - 1}\Big). 
\end{align*}
Note that the inequality $(a)$ holds regardless of whether $\alpha + \beta - 1 \geq 0$ or not. 
Combining the above results, the inequality in \eqref{eq:delay-upp-bd}, and $A_k \geq \frac{a_0}{\beta+1} k^{\beta+1}$, we can conclude that there exists a constant $C_{\text{sub}}$ independent of the iteration index $k$ such that
\begin{align*}
\Phi(y_k) - \Phi(x_*) \leq C_{\text{sub}} \cdot \frac{1}{k^{1 - \alpha}}, \forall k \in \nset{}{+}. 
\end{align*}
The specific formulation of $C_{\text{sub}}$ is omitted here, whose value depends on the diameter of the feasible set $\mathcal{X}$, the number of players, the parameters of the delay function, and $\beta$ chosen. 
\end{proof}

\subsection{Convergence Rate Analysis with Linear and Superlinear Feedback Delays}\label{pf:thm:convg-rt-spl-delay}
\begin{proof}
(Proof of Corollary~\ref{coro:convg-rt-lnr-delay})
By leveraging the similar arguments in the proof of Theorem~\ref{thm:convg-rt-delay}, 
\begin{align*}
& \Phi(y_k) - \Phi(x_*) \leq G_k = \frac{A_1}{A_k} G_1 + \frac{1}{A_k}\sum_{t=2}^{k} E_t \\
& \leq \frac{1}{A_k}D_{\psi}(x_*, x_0) + \frac{1}{A_k}\sum_{t=2}^{k} E^{e}_{t} + \frac{1}{A_k} \sum_{t=1}^{k} E^{\eta}_t, 
\end{align*}
where $E^{e}_{t}$ and $E^{\eta}_t$ are defined in Lemma~\ref{le:cohen}. 
As has been discussed in the proof of Theorem~\ref{thm:convg-rt-no-delay}, the error $E^e_t$ can be upper bounded as follows:
\begin{align*}
    E^e_k \leq \frac{1}{2}\Big(\frac{(a_k)^2L}{A_k} - \mu_*\Big) \norm{v_k - v_{k-1}}^2. 
\end{align*}
Since $a_k = a_0/k$ and $A_k \geq a_0 \log(k)$, for arbitrary $a_0$, $L$, and $\mu_*$, there always exists an index $K_*$ such that $(K_*)^2\log(K^*) \geq a_0L/\mu_*$. 
The summation over $(E^e_t)_{t=2}^{k}$ can be bounded in the following way:
\begin{align*}
\sum_{t=2}^{k} E^e_t \leq \frac{1}{2}(D_\mathcal{X})^2\sum_{t=2}^{K^*-1} \Big(\frac{a_0L}{t^2\log(t)} - \mu_*\Big) = E^e_*, \forall k \in \nset{}{+}. 
\end{align*}
For the remaining terms, we have 
\begin{align*}
& \sum_{t=1}^{k}a_t \langle \eta_t, x_* - v_t\rangle 
= \sum_{i \in \playerN}\sum_{t=1}^{k} a_t\langle \eta^i_t, x^i_* - v^i_t\rangle. 
\end{align*}
Note that following the initialization in Algorithm~\ref{alg:updt-delay} ensures that $s^i(k) \geq 1$ for all $k \in \nset{}{+}$ and $\eta_1 = 0$. 
Applying similar arguments in the proof of Theorem~\ref{thm:convg-rt-delay} yields:
\begin{align*}
& \sum_{t=1}^{k} a_t\langle \eta^i_t, x^i_* - v^i_t\rangle \leq L(D_{\mathcal{X}})^2\sum_{t=2}^{k} a_t \sum_{\ell=s^i(t)+1}^{t} \frac{a_\ell + a_{\ell-1}}{A_{\ell}} \\
& \leq L(D_{\mathcal{X}})^2\sum_{t=2}^{k} a_t \Big(\sum_{\ell=s^i(t)+1}^{(D+1)(s^i(t)+1)} \frac{a_\ell + a_{\ell-1}}{A_{\ell}}\Big) \\
& \leq 2L(D_{\mathcal{X}})^2\sum_{t=2}^{k} a_t \Big(\sum_{\ell=s^i(t)+1}^{(D+1)(s^i(t)+1)} \frac{2}{\ell \log(\ell)}\Big) \\
& \leq 4L(D_{\mathcal{X}})^2\sum_{t=2}^{k} a_t \Big(\int_{s^i(t) + 1}^{(D+1)(s^i(t)+1)} \frac{1}{\ell \log(\ell)}d\ell \\
& \qquad + \frac{1}{(s^i(t) + 1)\log(s^i(t) + 1)}\Big) \\
& = 4L(D_{\mathcal{X}})^2\sum_{t=2}^{k} a_t \Big(\log\big(\frac{\log(D+1) + \log(s^i(t)+1)}{\log(s^i(t) + 1)}\big)\\
& \qquad + \frac{1}{(s^i(t) + 1)\log(s^i(t) + 1)}\Big) \\
& \overset{(a)}{\leq} 4L(D_{\mathcal{X}})^2a_0 \sum_{t=2}^{k} \frac{1}{t}\Big( \frac{\log(D+1)}{\log(s^i(t)+1)} + \frac{1/\log(s^i(t) + 1)}{(s^i(t) + 1)}\Big) \\
& \overset{(b)}{\leq} 4L(D_{\mathcal{X}})^2a_0 \sum_{t=2}^{k} \frac{1}{t}\underbrace{\Big( \frac{\log(D+1)}{\log(\max(2, \frac{t}{D+1}))} + \frac{(D+1)/\log 2}{t} \Big)}_{\leq M^{\prime}/\log(t)} \\
& \leq 4L(D_{\mathcal{X}})^2a_0M^\prime \Big(\frac{1}{2\log2} + \int_{2}^{k} \frac{1}{\tau \log(\tau)}d\tau\Big) \\
& = 4L(D_{\mathcal{X}})^2a_0M^\prime\Big(\log(\log(k)) + \frac{1}{2\log2} - \log\log2\Big),  
\end{align*}
where $(a)$ is a consequence of the fact that $\log(1 + x) \leq x$ for all $x > -1$; 
for $(b)$, we use that $s^i(t) + 1 \geq \max\{2, t/(D+1)\}$, and note that the components inside the parentheses are dominated by the first part and can be upper bounded by $M^{\prime}\log(t)$ for some constant $M^{\prime}$ depending on $D$. 
Consequently, we have 
\begin{align*}
\Phi(y_k) - \Phi(x_*) = O(\frac{\log \log(k)}{\log(k)}).
\end{align*}
\end{proof}

\begin{proof}
(Proof of Corollary~\ref{coro:convg-rt-sup-delay})
Applying the same arguments in the proof of Corollary~\ref{coro:convg-rt-lnr-delay} yields:
\begin{align*}
\Phi(y_k) - \Phi(x_*) \leq \frac{1}{A_k}\big(D_{\psi}(x_*, x_0) + E^e_*\big) + \frac{1}{A_k}\sum_{t=1}^{k}E^\eta_t. 
\end{align*}
 
Noting that for $k \in \nset{}{+}$, we have 
\begin{align*}
& A_k = \sum_{t=2}^{k+1} \frac{a_0}{t\log(t)} \geq \int_{2}^{k+1} \frac{a_0}{t\log(t)} dt \\
& \geq a_0(\log\log(k+1) - \log\log(2)) \geq a_0\log\log(k+1). 
\end{align*}
Then the summation of $E^\eta_t$ can be upper bounded as follows:
\begin{align*}
& \sum_{t=2}^{k} a_t\langle \eta^i_t, x^i_* - v^i_t\rangle \leq L(D_{\mathcal{X}})^2\sum_{t=2}^{k} a_t \sum_{\ell=s^i(t)+1}^{t} \frac{a_\ell + a_{\ell-1}}{A_{\ell}} \\
& \leq 2L(D_{\mathcal{X}})^2   \cdot \\
& \qquad \sum_{t=2}^{k} \frac{a_0}{(t+1)\log(t+1)}\Big(\sum_{\ell = s^i(t) + 1}^{t}\frac{1}{\ell \log(\ell)\log\log(\ell+1)}\Big)
\end{align*}
Lemma~\ref{le:delay-bd} suggests that $(D+1)(s^i(t)+1)^{\alpha} \geq s^i(t)+1 + D(s^i(t)+1)^\alpha \geq t$ and $s^i(t) + 1 \geq (t/(D+1))^{1/\alpha}$. Then,
\begin{align*}
& \sum_{\ell = s^i(t) + 1}^{t} \frac{1}{\ell \log(\ell)\log\log(\ell+1)} \\
& \leq \sum_{\ell = s^i(t) + 1}^{(D+1)(s^i(t)+1)^{\alpha}} \frac{1}{\ell \log(\ell)\log\log(\ell+1)} \\
& \leq \underbrace{\int_{s^i(t) + 2}^{(D+1)(s^i(t)+2)^{\alpha}}\frac{1}{\ell \log(\ell)\log\log(\ell)} d\ell}_{(\romannum{1})} + \\
& \qquad \underbrace{\frac{2}{(s^i(t) + 1)\log 2 \log\log3}}_{(\romannum{2})}. \\
\end{align*}
For $t \geq 2$, we investigate $(\romannum{1})$ and $(\romannum{2})$ separately below:
\begin{align*}
& (\romannum{1}) = \log\Big( \frac{\log\log\big((D+1)(s^i(t)+2)^{\alpha}\big)}{\log\log(s^i(t)+2)} \Big) \\
& = \log\Big( \frac{\log\big(\log(D+1) + \alpha\log(s^i(t)+2)\big)}{\log\log(s^i(t)+2)} \Big) \\
& = \log \Big( 
\frac{\log\big(
1 + \frac{\log(D+1)}{\alpha\log(s^i(t)+2)}\big) + \log\big(\alpha\log(s^i(t)+2)\big)
}
{\log\log(s^i(t)+2)}
\Big) \\
& \overset{(a)}{\leq} \log\Big(1 + \frac{\log \alpha + \frac{\log(D+1)}{\alpha\log(s^i(t)+2)}}{\log\log(s^i(t)+2)}\Big) \\
& \overset{(b)}{\leq} \frac{\log(\alpha) + \frac{\log(D+1)}{\alpha\log 3}}{\log\log(s^i(t)+2)} \\
& \overset{(c)}{\leq} \frac{\log(\alpha) + \frac{\log(D+1)}{\alpha\log 3}}{
\max\{\log\log 3, \log\big(\log(t + 1) - \log(D+1)\big) - \log\alpha\}
}, \\
& (\romannum{2}) \overset{(d)}{\leq} \frac{2}{\log2\log\log3}\Big(\frac{D+1}{t}\Big)^{1/\alpha}, 
\end{align*}
where $(a)$ and $(b)$ are the results of the inequality $\log(1 + x) \leq x$ for all $x > -1$; 
$(d)$ directly follows from Lemma~\ref{le:delay-bd}, which suggests $(D+1)(s^i(t) + 1)^{\alpha} > t$; 
based on the fact that $x^{\alpha}$ is convex on $[0, +\infty)$, we can arrive at $(c)$ by noting that $(s^i(t) + 2)^{\alpha} \geq (s^i(t) + 1)^{\alpha} + 1^{\alpha} \geq \frac{t}{D+1} + 1 \geq \frac{t+1}{D+1}$. 
Note that $(\romannum{2})$ decays much faster than $(\romannum{1})$. 
Thus there exists a constant $M^{\star}$ that depends on $D$ and $\alpha$, such that $(\romannum{1}) + (\romannum{2}) \leq M^{\star}/\log\log(t+1)$ for all $t \geq 2$. 
Using this fact, the evolution of $\frac{1}{A_k}\sum_{t=1}^{k}E^\eta_t$ can be characterized as follows:
\begin{align*}
& \frac{1}{A_k}\sum_{t=1}^{k}E^\eta_t \leq \frac{2L(D_{\mathcal{X}})^2}{ \log\log(k+1)} \cdot \\
& \qquad\qquad \Big(\sum_{t=2}^{k}\frac{1}{(t+1)\log(t+1)} \cdot \frac{M^\star}{\log\log(t+1)} \Big) \\
& \leq \frac{2L(D_{\mathcal{X}})^2M^{\star}}{\log\log(k+1)} (\int^{k+1}_3 \frac{1}{\tau\log\tau\log\log\tau}d\tau + \frac{1/\log\log3}{3\log3}) \\
& \leq \frac{2L(D_{\mathcal{X}})^2M^{\star}}{\log\log(k+1)}(\log\log\log(k+1) + 6), 
\end{align*}
where the last inequality simply follows from the result of the integral and the constant $-\log\log\log3 + 1/(3\log3\log\log3) < 6$. 
Consequently, we conclude that
\begin{align*}
\Phi(y_k) - \Phi(x_*) = O(\frac{\log\log\log(k)}{\log\log(k)}). 
\end{align*}
\end{proof}